\documentclass[11pt]{amsart}

\usepackage[a4paper,margin=1in]{geometry}
\usepackage[T1]{fontenc}
\usepackage{lmodern}
\usepackage{microtype}
\usepackage{amsmath,amssymb,amsthm,mathtools}
\usepackage{booktabs,array}
\usepackage[shortlabels]{enumitem}
\usepackage[hidelinks]{hyperref}

\newtheorem{theorem}{Theorem}[section]
\newtheorem{proposition}[theorem]{Proposition}
\newtheorem{lemma}[theorem]{Lemma}
\newtheorem{corollary}[theorem]{Corollary}
\theoremstyle{definition}
\newtheorem{definition}[theorem]{Definition}

\theoremstyle{remark}

\newcommand{\CH}{\operatorname{CH}}
\newcommand{\Ch}{\operatorname{Ch}}
\newcommand{\CHW}{\widetilde{\operatorname{CH}}}
\newcommand{\KM}{\mathbf K^{\mathrm M}}
\newcommand{\KMW}{\mathbf K^{\mathrm{MW}}}
\newcommand{\I}{\mathbf I}
\newcommand{\W}{\operatorname W}
\newcommand{\GW}{\operatorname{GW}}
\newcommand{\Pic}{\operatorname{Pic}}
\newcommand{\Hom}{\operatorname{Hom}}
\newcommand{\im}{\operatorname{im}}
\newcommand{\Spec}{\operatorname{Spec}}
\newcommand{\F}{\mathbb F}
\newcommand{\Z}{\mathbb Z}
\newcommand{\PP}{\mathbb P}
\newcommand{\RZ}{\mathbb R\mathcal Z}
\newcommand{\sa}{\mathrm{sa}}
\newcommand{\RS}{\mathrm{RS}}
\newcommand{\wtPic}{\overline{\Pic}}

\title{A Fan Algorithm for Chow-Witt Rings of Smooth Projective Toric Varieties}
\author{Haoyang Liu, Tianle Liu and Guorui Xu}
\date{}
\address{University of California, Santa Barbara}
\email{haoyangliu@ucsb.edu}
\address{University of Southern California}
\email{tianleli@usc.edu}
\address{Fudan University}
\email{xugr@fudan.edu.cn}

\begin{document}

\begin{abstract}
Let $R$ be a real closed field and let $X_\Sigma$ be a smooth projective split toric
variety.  For an explicitly chosen basis rigidification $\mathfrak r$ of its Picard
grading, we prove that the pair $(\Sigma,\mathfrak r)$ determines the resulting total
Chow--Witt ring and give a terminating finite algorithm for all groups, products,
twists, and forgetful maps.
The ring is identified with an explicit fibre product of the
Picard-graded diagonal $\I$-cohomology ring and the integral inverse images of twisted
Bockstein kernels over the mod-$2$ Chow ring.  Using real cycle classes, we identify the
first factor with the cohomology of the real toric variety with all sign local systems.
We construct the mod-$2$ cycle map on invariant divisors as explicit deck-transition
cocycles and obtain finite signed boundary and product matrices directly from the fan.

\end{abstract}

\maketitle
\tableofcontents

\section{Introduction}

Chow--Witt groups refine ordinary Chow groups by retaining quadratic and orientation
data.  For a smooth scheme $X$ over a field of characteristic different from $2$, the
natural multiplicative object is not the untwisted group alone but the collection
\[
 \bigl\{\CHW^p(X,L)\bigr\}_{p,L},\qquad
 \CHW^p(X,L)\times\CHW^q(X,M)
 \longrightarrow\CHW^{p+q}(X,L\otimes M).
\]
An ordinary ring formed by summing only over $\Pic(X)/2$ is not canonical: it requires
representatives and coherent square-cancellation data.  We keep this rigidification
visible throughout.  For the toric varieties considered here, $\Pic(X_\Sigma)$ is free,
and an ordered Picard basis gives the required data explicitly.
The twist is essential in toric geometry.  We write $(L,a)$ for the graded line bundle
with underlying line bundle $L$ and degree $a$, and use degree-zero graded line bundles
throughout the rigidified total ring.  Thus a line bundle $L$ means $(L,0)$ unless
another degree is displayed.  This convention is distinct from
the natural Euler-class grading: the Euler class of a line bundle $\mathcal O(D)$ has
graded twist $(\mathcal O(-D),1)$.  Its underlying line-bundle class agrees with that of
$\mathcal O(D)$ in $\Pic(X)/2$, but neither the dual twist nor its graded degree may be
discarded before the convention has been fixed.

The basic reconstruction of Chow--Witt theory uses ordinary Chow groups and
$\I$-cohomology.  It is tempting to regard this as an automatic fibre-product formula
for every smooth scheme and then to replace the Bockstein kernel by the kernel of a
motivic Steenrod square.  Neither step is formal.  Hornbostel and Wendt prove that the
canonical map to the fibre product is always surjective, and that it is injective under
an additional hypothesis, for example the absence of nonzero $2$-torsion in the Chow
groups \cite[Proposition~2.11]{HornbostelWendt}.  Moreover, the identity
\[
 \rho\beta_L=\operatorname{Sq}^2_L
 =\operatorname{Sq}^2+c_1(L)(-)
\]
only gives $\ker\beta_L\subseteq\ker\operatorname{Sq}^2_L$ unless reduction is
injective on $\im\beta_L$ \cite[Theorem~3.4.1]{AsokFasel}.

For smooth projective split toric varieties both issues can be handled cleanly.
Their Chow groups are free abelian, so the fibre-product theorem applies.  Over a real
closed field, the real cycle class theorem of Hornbostel--Wendt--Xie--Zibrowius
identifies diagonal $\I$-cohomology, with all line-bundle twists and products, with
semialgebraic cohomology of the real locus with the corresponding sign local systems
\cite[Corollary~4.4 and Theorem~5.7]{RealCycle}.  The remaining task is to turn that
topological description into a finite construction from the fan without hiding any
multiplicative or twisting input.

Let $K$ be the simplicial complex of cones of a smooth projective fan with $r$ rays,
and let
\[
 \lambda_2:(\Z/2)^r\longrightarrow(\Z/2)^n
\]
be the mod-$2$ characteristic map.  Put $K_0=\ker\lambda_2$.  The real moment-angle
complex $\RZ_K=(D^1,S^0)^K$ carries the coordinate action of $(\Z/2)^r$.  Smoothness
implies that $K_0$ acts freely: the stabilizer of a point is supported on a face
$\sigma\in K$, and the primitive ray vectors indexed by $\sigma$ are linearly
independent modulo $2$.  Furthermore,
\[
 \Hom(K_0,\{\pm1\})\cong\Pic(X_\Sigma)/2.
\]
Thus every algebraic line-bundle twist is represented by a sign character of the
cover $\RZ_K\to\RZ_K/K_0$.

After a $K_0$-invariant subdivision, normalized simplicial chains on $\RZ_K$ are
finite free $\Z[K_0]$-modules.  Naturality of the Alexander--Whitney diagonal makes it
$K_0$-equivariant.  Hence
\[
 C^*_\chi(\Sigma)=
 \Hom_{\Z[K_0]}(C_*(\RZ_K),\Z_\chi)
\]
computes cohomology with the local system $\Z_\chi$, and the equivariant diagonal
gives every pairing
\[
 C^*_\chi(\Sigma)\otimes C^*_\psi(\Sigma)
 \longrightarrow C^*_{\chi\psi}(\Sigma).
\]
The same finite complexes compute the coefficient reductions and the twisted integral
Bocksteins.  This leads to the main result.

The construction is effective at every stage.  From the rays and cones and the chosen
rigidification it produces
finite integer matrices for the Picard-graded $\I$-cohomology ring, the divisor cycle
map, coefficient reductions, and Bocksteins.  A supplementary implementation checks
the fan-to-cochain and divisor-cycle layers; no proof depends on executing it.

\begin{theorem}
Let $R$ be real closed and let $X_\Sigma/R$ be a smooth projective split toric
variety, equipped with an ordered Picard-basis rigidification $\mathfrak r$.  The
rigidified total ring $\CHW^\bullet_{\mathfrak r}(X_\Sigma,\Pic/2)$, including all
products and forgetful maps, is computable by a terminating finite algorithm from the
rays and cones of $\Sigma$ and the displayed rigidification.  Explicitly, the ring is
the fibre product of the signed moment-angle
cohomology ring and the integral inverse images of the kernels of the corresponding
twisted Bocksteins over $\Ch^*(X_\Sigma)[\Pic/2]$.
\end{theorem}

The exact use of the Bockstein is a feature of the theorem.  When reduction is known
to be injective on the Bockstein image, the right factor may be computed as a twisted
Steenrod kernel; without this verification we do not make that replacement.

We execute the construction for the Hirzebruch surface $\mathbb F_a$.  The real locus
is a torus when $a$ is even and a Klein bottle when $a$ is odd.  We write down the
four signed cellular complexes, their reduction maps, and their complete
Picard-graded multiplication.  This gives a full rigidified fibre-product description for every
$a\geq0$, not merely additive untwisted groups.

The real-closed hypothesis is not necessary in every individual case: the cellular
real-cycle comparison extends to the broader fields described in
\cite[Remark~5.10 and Lemma~5.11]{RealCycle}.  The projective-line calculation below
instead proves the precise uniform boundary needed here: without retaining additional
arithmetic input, the real-closed hypothesis cannot simply be replaced by an arbitrary
perfect ordered field.  For an ordered subfield $k\subset R$, Fasel's projective bundle
calculation gives
\[
 H^1(\PP^1_k,\I^1)\cong\W(k),\qquad
 H^1(\PP^1_k,\I^1(\mathcal O(1)))\cong\W(k)/I(k)\cong\Z/2.
\]
The corresponding groups on $S^1$ are $\Z$ and $\Z/2$.  The twisted sector happens
to agree, but the untwisted comparison is the chosen signature
$\W(k)\to\Z$ and generally has a nonzero kernel.  No field-independent signed
complex can recover that arithmetic information.

The paper is organized as follows.  Section~\ref{sec:fibre} records the exact
fibre-product statement and separates Bockstein and Steenrod kernels.
Section~\ref{sec:toric} establishes the toric algebraic input.  Section~\ref{sec:signed}
constructs signed cochains and their products.  Section~\ref{sec:realcycle} proves the
main theorem over real closed fields.  Section~\ref{sec:hirzebruch} contains the
complete Hirzebruch-surface computation.  Section~\ref{sec:boundary} gives the
projective-line obstruction and states the remaining arithmetic boundary.

\subsection*{Scope.}
Theorem~\ref{thm:main} is stated for smooth projective fans for three independent
reasons: the Stanley--Reisner Chow presentation is used in its classical complete
range; projectivity supplies a filtrable affine-cell decomposition; and the real locus
has the compact small-cover/moment-angle model used in Section~\ref{sec:signed}.
The theorem extends verbatim to a smooth complete split toric variety whenever these
three inputs are separately supplied.  For a noncomplete fan one must replace the
compact real moment-angle quotient by an appropriate noncompact model and prove the
required cellular and Chow-theoretic statements; we make no such assertion here.

Likewise, shellability gives useful additive information.  Projective toric fans are
shellable by Bruggesser--Mani, and Liu--Peng's preprint then provides an additive
MW-motivic normal form \cite{LiuPeng}.  A promotion of that normal form to a total
Chow--Witt ring would require a multiplicative diagonal, all Picard-twisted
coefficient objects, and their reduction maps.  The signed construction supplies
these data after real realization over a real closed field, not over a general field.

\subsection*{Relation to previous work.}
The algebraic contribution has three parts.  First, Theorem~\ref{thm:main} gives a
single rigidified Picard-graded fibre-product presentation for all twists and proves that every
term and product is computable from the fan together with the rigidification.  Second,
Proposition~\ref{prop:explicitcl2} makes the mod-$2$ cycle map constructive in the
toric divisor presentation, so that exact Bockstein kernels, rather than only Steenrod
upper bounds, can be computed.  Third, Section~\ref{sec:hirzebruch} determines every
twist sector, product, and forgetful map for the complete family $\mathbb F_a$.

Jiang studied Chow--Witt rings of toric surfaces in his doctoral thesis
\cite{JiangThesis}.  The treatment here is independent and self-contained; it addresses
all classes in $\Pic/2$ simultaneously, retains exact Bocksteins, and derives the
surface calculation from a general construction for smooth projective fans of arbitrary
dimension.

Franz gives an equivariant dga computing the cohomology ring of a real toric space over
arbitrary constant coefficient rings \cite{Franz}.  Our signed complexes supply the
simultaneous rank-one local-system extension required by the Picard grading, together
with the reductions and connecting maps needed by the Chow--Witt square.  Fan gives an
additive integral formula and a ring formula modulo the ideal of order-$2$ elements
\cite{Fan}; the present fibre product retains the $2$-primary data.  Liu and Peng
construct cellular $\mathbb A^1$-homology complexes and an additive normal form for
pure shellable fans \cite{LiuPeng}.  Their construction is complementary to the
all-twist multiplicative diagonal used here.  The supplementary certificate records
the finite fan-to-cochain construction and independently checks its defining identities.

\subsection*{Acknowledgment}
The first author would like to thank Fangzhou Jin and Keyao Peng for helpful discussions.

\section{The fibre-product reconstruction}
\label{sec:fibre}

Throughout this section, $k$ is a perfect field of characteristic different from $2$
and $X$ is a smooth $k$-scheme.  We write $\Ch^p(X)=\CH^p(X)/2$.

\subsection{Twisted Chow--Witt groups}

For $x\in X^{(p)}$, put
$\omega_x=\det(\mathfrak m_x/\mathfrak m_x^2)^\vee$.  The degree-$p$ term of the
Rost--Schmid complex with Milnor--Witt coefficients and twist $L$ is
\[
 C^p_{\RS}(X,\KMW_q,L)=
 \bigoplus_{x\in X^{(p)}}
 K^{\mathrm{MW}}_{q-p}(k(x),\omega_x\otimes L_x).
\]
Its differential is the sum of the twisted residue maps, and
\[
 \CHW^p(X,L)=H^p(C^\bullet_{\RS}(X,\KMW_p,L)).
\]
We use Morel's and Fasel's conventions for twists and residues
\cite[Chapter~5]{Morel}\cite{FaselLectures}.  A twist written simply as $L$ is, by
definition in this paper, the graded line bundle $(L,0)$.  Consequently tensor product
of twists gives
\[
 \CHW^p(X,L)\times\CHW^q(X,M)
 \longrightarrow\CHW^{p+q}(X,L\otimes M).
\]
Changing a twist by a square yields a change-of-twist isomorphism only after the square
root and the relevant line-bundle isomorphism have been specified.  Consequently a
direct sum indexed by $\wtPic(X):=\Pic(X)/2$ is not, without further data, a canonical
ring.  The degree of a graded line bundle is an additional genuine part of the
commutativity convention; see
\cite[Remark~2.3 and Section~2.3]{WendtGrassmannians}.  Our degree-zero convention is
fixed once and for all and is part of every rigidification below.

\begin{definition}
\label{def:rigidification}
Assume that $\Pic(X)$ is free of rank $\varrho$.  A \emph{basis rigidification}
$\mathfrak r$ consists of an ordered basis $B_1,\ldots,B_\varrho$ of $\Pic(X)$ and
line bundles $\mathcal B_i$ representing the $B_i$, each regarded as the graded line
bundle $(\mathcal B_i,0)$.  For
$\epsilon=(\epsilon_i)\in(\Z/2)^\varrho$ put
\[
 L_\epsilon=\bigotimes_{i=1}^{\varrho}\mathcal B_i^{\otimes\epsilon_i},
 \qquad
 Q_{\epsilon,\delta}=
 \bigotimes_{i=1}^{\varrho}\mathcal B_i^{\otimes\epsilon_i\delta_i},
\]
where all tensor products are taken in the displayed order.  The standard
associativity and symmetry constraints give a coherent isomorphism
\[
 \theta_{\epsilon,\delta}:L_\epsilon\otimes L_\delta
 \xrightarrow{\ \cong\ }
 L_{\epsilon+\delta}\otimes Q_{\epsilon,\delta}^{\otimes2}.
\]
If $s_Q$ denotes square cancellation
$\CHW^m(X,N\otimes Q^{\otimes2})\xrightarrow{\cong}\CHW^m(X,N)$, define
\[
 \CHW^\bullet_{\mathfrak r}(X,\wtPic)=
 \bigoplus_{p,\epsilon}\CHW^p(X,L_\epsilon),
 \qquad
 x\star_{\mathfrak r}y=
 s_{Q_{\epsilon,\delta}}
 \bigl((\theta_{\epsilon,\delta})_*(xy)\bigr).
\]
\end{definition}

For each $i$, regarding $\epsilon_i,\delta_i,\gamma_i$ as the integers $0$ or $1$, one
has
\[
 \epsilon_i\delta_i+ (\epsilon_i+\delta_i\bmod 2)\gamma_i
 =\delta_i\gamma_i+\epsilon_i(\delta_i+\gamma_i\bmod 2).
\]
Consequently the two parenthesizations of a triple product cancel canonically the same
ordered square factor.  Mac Lane coherence for the tensor product, together with the
functoriality and multiplicativity of change of twist by a square, then proves that
$\star_{\mathfrak r}$ is associative and unital.  If
$x\in\CHW^p(X,L_\epsilon)$ and $y\in\CHW^q(X,L_\delta)$ are homogeneous, then, after
using the symmetry isomorphism to identify the two target twists,
\begin{equation}
 x\star_{\mathfrak r}y
 =\langle-1\rangle^{pq}y\star_{\mathfrak r}x.
 \label{eq:rigidcommutativity}
\end{equation}
There is no further correction from the twist labels because every chosen graded line
bundle has degree zero.  Thus
$\CHW^\bullet_{\mathfrak r}(X,\wtPic)$ is a
$\Z\oplus(\Z/2)^\varrho$-graded ring.  Its component in degree $(0,0)$ is
$\CHW^0(X,\mathcal O_X)$, and pullback along $X\to\Spec k$ makes it a
$\GW(k)$-algebra; it need not equal $\GW(k)$ for an arbitrary smooth $X$.  A different
basis rigidification generally gives no canonical identification with this ring.
Without freeness or a chosen rigidification, all statements below are to be read sector
by sector, or as statements graded by the Picard groupoid rather than as statements
about an ordinary total ring.

\subsection{Bocksteins and the exact right factor}

The cartesian square of homotopy modules relating $\KMW_*$, $\KM_*$, and $\I^*$
gives, after twisting by $L$, the B\"ar sequence
\begin{equation}
 \cdots\longrightarrow H^p(X,\I^{p+1}(L))
 \longrightarrow H^p(X,\I^p(L))
 \xrightarrow{\rho_L}\Ch^p(X)
 \xrightarrow{\beta_L}H^{p+1}(X,\I^{p+1}(L))
 \longrightarrow\cdots .
 \label{eq:bar}
\end{equation}
Define
\begin{equation}
 K_L^p(X)=\ker\!\left(
 \CH^p(X)\xrightarrow{\bmod 2}\Ch^p(X)
 \xrightarrow{\beta_L}H^{p+1}(X,\I^{p+1}(L))
 \right).
 \label{eq:exactK}
\end{equation}
If $\Pic(X)$ is free and $\mathfrak r$ is a basis rigidification, set
\[
 A_{\mathfrak r}(X)=\bigoplus_{p,\epsilon}H^p(X,\I^p(L_\epsilon)),\quad
 C_{\mathfrak r}(X)=\Ch^\bullet(X)[(\Z/2)^\varrho],\quad
 K_{\mathfrak r}(X)=\bigoplus_{p,\epsilon}K_{L_\epsilon}^p(X).
\]
The product on $A_{\mathfrak r}(X)$ uses the same maps
$\theta_{\epsilon,\delta}$ and square cancellations as
Definition~\ref{def:rigidification}.  The twist label in $C_{\mathfrak r}(X)$ is
formal: each sector is a copy of $\Ch^\bullet(X)$ and labels add.  The product on
$K_{\mathfrak r}(X)$ is the restriction of ordinary Chow multiplication with these
formal labels; its closure also follows from the theorem below.

\begin{theorem}
\label{thm:fibre}
Assume that $\CH^p(X)$ has no nonzero $2$-torsion for every $p$.  For every line bundle
$L$, change of coefficients induces a canonical isomorphism
\begin{equation}
 \CHW^p(X,L)\xrightarrow{\ \cong\ }
 H^p(X,\I^p(L))\times_{\Ch^p(X)}K_L^p(X).
 \label{eq:fibresector}
\end{equation}
These isomorphisms are multiplicative with respect to tensor products of actual
twists.  If, in addition, $\Pic(X)$ is free and $\mathfrak r$ is a basis
rigidification, they induce an isomorphism of
$\Z\oplus(\Z/2)^\varrho$-graded rings
\begin{equation}
 \CHW^\bullet_{\mathfrak r}(X,\wtPic)
 \xrightarrow{\ \cong\ }
 A_{\mathfrak r}(X)\times_{C_{\mathfrak r}(X)}K_{\mathfrak r}(X).
 \label{eq:fibre}
\end{equation}
The two maps to $C_{\mathfrak r}(X)$ are $\rho_L$ and ordinary reduction modulo $2$ in each
twist sector.
\end{theorem}

\begin{proof}
Fix $p$ and $L$ and abbreviate the four Rost--Schmid complexes by
\[
 \mathcal C_{\mathrm{MW}}(L)=C^\bullet_{\RS}(X,\KMW_p,L),\quad
 \mathcal C_I(L)=C^\bullet_{\RS}(X,\I^p,L),
\]
\[
 \mathcal C_M=C^\bullet_{\RS}(X,\KM_p),\qquad
 \overline{\mathcal C}_M=C^\bullet_{\RS}(X,\KM_p/2).
\]
The fundamental square of homotopy modules, after tensoring orientation lines with
$L$, gives a termwise cartesian square
\begin{equation}
\begin{array}{ccc}
 \mathcal C_{\mathrm{MW}}(L)&\longrightarrow&\mathcal C_I(L)\\
 \big\downarrow&&\big\downarrow\\
 \mathcal C_M&\longrightarrow&\overline{\mathcal C}_M.
\end{array}
\label{eq:complexpullback}
\end{equation}
At a codimension-$q$ point this is the cartesian square of
$K^{\mathrm{MW}}_{p-q}$, $I^{p-q}$, $K^{\mathrm M}_{p-q}$, and
$K^{\mathrm M}_{p-q}/2$, all with the residue-field orientation line; compatibility
of the four residue homomorphisms makes \eqref{eq:complexpullback} a square of
complexes, not merely of terms.

The resulting cohomology comparison is the twisted analogue of
\cite[Proposition~2.11]{HornbostelWendt}.  It is stated, including the arbitrary
line-bundle twist and the total-ring product, in
\cite[Section~2.3]{WendtGrassmannians}: the canonical map
\[
 \CHW^p(X,L)\longrightarrow
 H^p(X,\I^p(L))\times_{\Ch^p(X)}K_L^p(X)
\]
is always surjective, and its kernel vanishes when $\CH^p(X)$ has no nonzero
$2$-torsion.  This proves \eqref{eq:fibresector}.

Finally, exterior products on the four complexes in \eqref{eq:complexpullback} commute
with the pullback along the diagonal and send the $(L,M)$ sectors to the
$L\otimes M$ sector \cite[Section~4]{FaselRing}.  These products also commute with
transport along $\theta_{\epsilon,\delta}$ and with square cancellation.  Applying
the same rigidification to all four corners therefore makes the sectorwise
isomorphisms assemble to \eqref{eq:fibre}.  In particular,
$K_{\mathfrak r}(X)$ is a subring: it is the image of the forgetful homomorphism
from $\CHW^\bullet_{\mathfrak r}(X,\wtPic)$.
\end{proof}

\subsection{When a Steenrod kernel is legitimate}

Asok--Fasel's twisted form of Totaro's theorem identifies the composite in
\eqref{eq:bar} as
\begin{equation}
 \rho_L\beta_L(z)=\operatorname{Sq}^2_L(z)
 =\operatorname{Sq}^2(z)+c_1(L)z
 \quad\text{in }\Ch^{p+1}(X).
 \label{eq:twistedsq}
\end{equation}
Consequently
\[
 \ker\beta_L\subseteq\ker\operatorname{Sq}^2_L,
\]
but equality requires an additional argument.

\begin{lemma}
\label{lem:sqcriterion}
For fixed $p$ and $L$, suppose that
\[
 \rho_L:H^{p+1}(X,\I^{p+1}(L))\longrightarrow\Ch^{p+1}(X)
\]
is injective on $\im\beta_L$.  Then
\[
 K_L^p(X)=
 \{z\in\CH^p(X):\operatorname{Sq}^2_L(\bar z)=0\}.
\]
Without this injectivity, the right-hand side is only an upper bound for
$K_L^p(X)$.
\end{lemma}

\begin{proof}
The inclusion from left to right follows from \eqref{eq:twistedsq}.  Conversely, if
$\operatorname{Sq}^2_L(\bar z)=0$, then $\rho_L\beta_L(\bar z)=0$.  Injectivity on
$\im\beta_L$ gives $\beta_L(\bar z)=0$, hence $z\in K_L^p(X)$.
\end{proof}

The algorithm below computes $\beta_L$ itself from an integral signed cochain
complex.  It therefore does not need the injectivity hypothesis in
Lemma~\ref{lem:sqcriterion}.

\section{Smooth projective toric input}
\label{sec:toric}

Let $\Sigma$ be a smooth projective fan of dimension $n$ in a lattice $N\cong\Z^n$.
Write $\Sigma(1)=\{\rho_1,\ldots,\rho_r\}$, let $v_i\in N$ be the primitive ray
generator of $\rho_i$, and let $D_i$ be the corresponding invariant divisor.  Put
$M=N^\vee$.

\subsection{Chow groups and twists}

Let $K=K_\Sigma$ be the simplicial complex on $[r]$ whose faces are the sets of rays
contained in a cone of $\Sigma$.  The Danilov--Jurkiewicz presentation is
\begin{equation}
 \CH^*(X_\Sigma)=
 \frac{\Z[x_1,\ldots,x_r]}{I_{\mathrm{SR}}+J_\Sigma},
 \label{eq:DJ}
\end{equation}
where $x_i=[D_i]$, $I_{\mathrm{SR}}$ is generated by
$\prod_{i\in S}x_i$ for $S\notin K$, and
\[
 J_\Sigma=
 \left(\sum_{i=1}^r\langle m,v_i\rangle x_i:m\in M\right)
\]
\cite[Section~5.2]{Fulton}.  Completeness is part of the hypothesis here; we make no
Stanley--Reisner assertion for arbitrary noncomplete fans.

The divisor sequence is
\begin{equation}
 0\longrightarrow M\xrightarrow{\lambda^\vee}\Z^r
 \longrightarrow\Pic(X_\Sigma)\longrightarrow0,
 \qquad
 m\longmapsto(\langle m,v_i\rangle)_i.
 \label{eq:divisorsequence}
\end{equation}
For a smooth complete toric variety the Picard group is free abelian.  Reducing
\eqref{eq:divisorsequence} modulo $2$ gives
\begin{equation}
 0\longrightarrow M/2\xrightarrow{\lambda_2^\vee}(\F_2)^r
 \longrightarrow\Pic(X_\Sigma)/2\longrightarrow0,
 \label{eq:picmod2}
\end{equation}
where
\[
 \lambda_2:(\F_2)^r\longrightarrow N/2,
 \qquad e_i\longmapsto v_i\bmod2.
\]
Because $\Sigma$ contains a smooth full-dimensional cone, $\lambda_2$ is surjective.

\begin{proposition}
\label{prop:cellular}
The variety $X_\Sigma$ has a filtration by closed subvarieties whose successive open
strata are affine spaces.  In particular, every $\CH^p(X_\Sigma)$ is free abelian,
and Theorem~\ref{thm:fibre} applies to every twist.
\end{proposition}

\begin{proof}
Choose a generic one-parameter subgroup of the split torus.  Its fixed points are the
finitely many torus fixed points indexed by maximal cones.  Since $X_\Sigma$ is smooth
and projective, the Bia\l ynicki--Birula decomposition is filtrable and each attracting
stratum is an affine space.  The localization sequence for Chow groups, inducting over
the filtration, expresses each Chow group as a free abelian group with one generator
for each cell of the appropriate codimension.  This also proves directly the
$2$-torsion hypothesis required in Theorem~\ref{thm:fibre}.
\end{proof}

\subsection{The real moment-angle cover}

Let $G=(\Z/2)^r$ act on
\[
 \RZ_K=(D^1,S^0)^K
 =\bigcup_{\sigma\in K}(D^1)^\sigma\times(S^0)^{[r]\setminus\sigma}
\]
by coordinate sign changes, and put $K_0=\ker\lambda_2$.  For a smooth projective
fan, the real toric space $\RZ_K/K_0$ is a finite model for the real locus
$X_\Sigma(\mathbb R)$; equivalently, it is the associated small cover.  This is a
special case of the real toric-space construction and of Franz's comparison with
smooth real toric varieties \cite[Corollary~4.6]{Franz}.

The following elementary lemma is the point at which smoothness enters the
topological construction.

\begin{lemma}
\label{lem:freeness}
The action of $K_0$ on $\RZ_K$ is free.
\end{lemma}

\begin{proof}
For $z=(z_1,\ldots,z_r)\in\RZ_K$, let
$\sigma(z)=\{i:z_i=0\}$.  By the definition of the polyhedral product,
$\sigma(z)\in K$.  A coordinate sign change fixes $z$ precisely in coordinates
belonging to $\sigma(z)$, so
\[
 \operatorname{Stab}_G(z)=G_{\sigma(z)}
 =\langle e_i:i\in\sigma(z)\rangle.
\]
The rays indexed by a face $\sigma\in K$ lie in a smooth cone.  Their primitive
generators extend to a $\Z$-basis of $N$, hence their reductions are linearly
independent in $N/2$.  Thus $\lambda_2|_{G_\sigma}$ is injective and
$K_0\cap G_\sigma=0$.  Therefore every $K_0$-stabilizer is trivial.
\end{proof}

\begin{lemma}
\label{lem:characters}
There is a natural isomorphism
\begin{equation}
 \Pic(X_\Sigma)/2\xrightarrow{\ \cong\ }
 \Hom(K_0,\{\pm1\}).
 \label{eq:characters}
\end{equation}
If $L=\mathcal O(\sum_i a_iD_i)$, the corresponding character is
\[
 \chi_L(g)=(-1)^{\sum_i a_ig_i},\qquad g=(g_i)\in K_0.
\]
The local system determined by the real line bundle $L(\mathbb R)$ on
$X_\Sigma(\mathbb R)$ is the sign local system associated with $\chi_L$.
\end{lemma}

\begin{proof}
Dualizing the exact sequence
$0\to K_0\to(\F_2)^r\xrightarrow{\lambda_2}N/2\to0$ gives
\[
 0\to M/2\xrightarrow{\lambda_2^\vee}(\F_2)^r
 \to\Hom(K_0,\F_2)\to0.
\]
Comparison with \eqref{eq:picmod2} proves the first assertion.  A divisor vector
$a=(a_i)$ acts on the Cox coordinate cover by the sign $(-1)^{a\cdot g}$.
Changing $a$ by a principal divisor adds an element in $\im\lambda_2^\vee$, which is
trivial on $K_0$.  Thus the associated bundle is
\[
 (\RZ_K\times\mathbb R)/K_0,
 \qquad g\cdot(z,t)=(gz,\chi_L(g)t),
\]
and its orientation local system is exactly $\Z_{\chi_L}$.
\end{proof}

\section{Signed cochains and products}
\label{sec:signed}

Let $Y=\RZ_K$ and $H=K_0$.  By Lemma~\ref{lem:freeness}, $Y\to Y/H$ is a
finite regular covering.  We now give a chain-level model that simultaneously computes
all sign local systems and all products between them.

\subsection{An equivariant diagonal}

The standard cubical cell structure on $Y$ need not carry a strictly equivariant
cellular diagonal in the form one first writes down.  We avoid this issue by passing to
a canonical invariant subdivision.

\begin{proposition}
\label{prop:equivdiag}
There is a finite $H$-simplicial complex $Y'$ equivariantly homeomorphic to a
subdivision of $Y$ such that:
\begin{enumerate}[(i)]
\item $H$ acts freely on the simplices of $Y'$;
\item $C_*(Y';\Z)$ is a bounded complex of finite free $\Z[H]$-modules;
\item the Alexander--Whitney map
\[
 \Delta_{\mathrm{AW}}:C_*(Y';\Z)\longrightarrow
 C_*(Y';\Z)\otimes C_*(Y';\Z)
\]
is $H$-equivariant.
\end{enumerate}
Its induced products on cohomology are independent of the chosen invariant
subdivision and agree with the canonical cup products with local coefficients.
\end{proposition}

\begin{proof}
Take the barycentric subdivision of the finite cubical complex $Y$ and, if necessary,
subdivide once more so that the action is simplicial without inversions.  Since the
action on points is free, the action on simplices is free.  Choosing one simplex in
each orbit identifies every simplicial chain group with a finite direct sum of copies of
$\Z[H]$.

The Alexander--Whitney diagonal is natural with respect to simplicial maps.  Hence
for $h\in H$,
\[
 \Delta_{\mathrm{AW}}\circ h_*
 =(h_*\otimes h_*)\circ\Delta_{\mathrm{AW}},
\]
which proves equivariance.  The comparison maps between subdivisions are equivariant
chain homotopy equivalences.  Naturality of the singular cup product, or the standard
acyclic-models comparison with singular chains, shows that the induced cohomology
product is the canonical one and is independent of the model.
\end{proof}

\subsection{All sign local systems at once}

For a character $\chi:H\to\{\pm1\}$, let $\Z_\chi$ be the left
$\Z[H]$-module on which $h$ acts by multiplication by $\chi(h)$.  Define
\begin{equation}
 C^*_\chi(\Sigma)=
 \Hom_{\Z[H]}(C_*(Y';\Z),\Z_\chi).
 \label{eq:signedcomplex}
\end{equation}

\begin{proposition}
\label{prop:signedring}
There are natural isomorphisms
\[
 H^p(C^*_\chi(\Sigma))
 \cong H^p(Y/H;\Z_\chi).
\]
The equivariant diagonal of Proposition~\ref{prop:equivdiag}, followed by
multiplication $\Z_\chi\otimes\Z_\psi\to\Z_{\chi\psi}$, induces cochain pairings
\[
 C^*_\chi(\Sigma)\otimes C^*_\psi(\Sigma)
 \longrightarrow C^*_{\chi\psi}(\Sigma)
\]
whose cohomology pairings are the cup products with local coefficients.  Therefore
\begin{equation}
 \mathcal H_\Sigma=
 \bigoplus_{p,\chi}H^p(C^*_\chi(\Sigma))
 \label{eq:totalH}
\end{equation}
is a $\Z\oplus\Hom(H,\{\pm1\})$-graded ring computed from a finite presentation by
integer matrices.
\end{proposition}

\begin{proof}
Because $Y'\to Y'/H$ is a regular finite cover and $C_*(Y';\Z)$ is a complex of free
$\Z[H]$-modules, equivariant cochains are precisely cellular cochains on the quotient
with the associated local system.  The assertion about products follows from
$H$-equivariance of the diagonal and the defining tensor rule for sign modules.
\end{proof}

Reduction $\Z_\chi\to\F_2$ forgets the sign action and gives maps
\begin{equation}
 r_\chi:H^p(C^*_\chi(\Sigma))\longrightarrow H^p(Y/H;\F_2).
 \label{eq:reduction}
\end{equation}
The exact coefficient sequence
\begin{equation}
 0\longrightarrow\Z_\chi\xrightarrow{\,2\,}\Z_\chi
 \longrightarrow\F_2\longrightarrow0
 \label{eq:topbocksteinsequence}
\end{equation}
also gives a twisted Bockstein
\begin{equation}
 \beta_\chi:H^p(Y/H;\F_2)\longrightarrow H^{p+1}(Y/H;\Z_\chi).
 \label{eq:topbockstein}
\end{equation}
Both maps are obtained by finite integer linear algebra on the cochain matrices.

\begin{corollary}
\label{cor:effective}
From the rays and cones of $\Sigma$ one can compute, by a terminating procedure:
\begin{enumerate}[(i)]
\item every group $H^p(Y/H;\Z_\chi)$ by Smith normal form;
\item every product between character sectors;
\item every reduction map $r_\chi$;
\item every twisted Bockstein $\beta_\chi$.
\end{enumerate}
\end{corollary}

\begin{proof}
The simplicial complex, the finite group action, the matrices of
\eqref{eq:signedcomplex}, and the Alexander--Whitney diagonal are finite and
combinatorially determined.  Smith normal form computes kernels, images, and
cokernels over $\Z$.  Lifting an $\F_2$ cocycle to an integral signed cochain and
dividing its coboundary by $2$ computes \eqref{eq:topbockstein}.
\end{proof}

\subsection{The matrices produced by the fan}

We make the word ``algorithm'' literal.  Give the fan by the integer ray matrix
$V=(v_1\ \cdots\ v_r)$ and the list of its maximal cones.  The faces of $K$ are then
obtained by taking subsets of the maximal cones, and row reduction of $V\bmod2$
computes a basis of $H=K_0$ and the character table of
Lemma~\ref{lem:characters}.

The open cubical cells of $Y=\RZ_K$ are indexed by pairs
\begin{equation}
 e(I,\varepsilon),\qquad
 I\in K,\quad \varepsilon\in\{\pm1\}^{[r]\setminus I},
 \label{eq:cubicalcells}
\end{equation}
where the coordinates in $I$ are open intervals and the remaining coordinates are
vertices.  Face incidence is obtained by replacing an interval coordinate by either
endpoint.  Thus the face poset, its barycentric subdivision $Y'$, and the $H$-action
are finite data obtained by comparisons and additions in $\F_2$.

Choose an ordered representative set $\mathcal S_d$ for the $H$-orbits of oriented
$d$-simplices of $Y'$.  If $s\in\mathcal S_d$, write its $j$th oriented face uniquely as
\begin{equation}
 \partial_js=\varepsilon_{s,j}h_{s,j}\,\overline{\partial_js},
 \qquad \varepsilon_{s,j}\in\{\pm1\},\quad h_{s,j}\in H,\quad
 \overline{\partial_js}\in\mathcal S_{d-1}.
 \label{eq:transport}
\end{equation}
For a character $\chi$, the signed boundary matrix is
\begin{equation}
 (B_d^\chi)_{t,s}
 =\sum_{\overline{\partial_js}=t}(-1)^j
   \varepsilon_{s,j}\chi(h_{s,j}),
 \qquad t\in\mathcal S_{d-1},\ s\in\mathcal S_d,
 \label{eq:signedmatrix}
\end{equation}
and the cochain differential is its transpose.  Formula
\eqref{eq:signedmatrix} is independent of the chosen orbit representatives up to
diagonal changes of basis by signs.

The Alexander--Whitney formula
\begin{equation}
 [c_0<\cdots<c_d]\longmapsto
 \sum_{i=0}^d[c_0<\cdots<c_i]\otimes[c_i<\cdots<c_d]
 \label{eq:AWmatrix}
\end{equation}
together with the transports \eqref{eq:transport} gives integer structure matrices
for every cochain pairing $C^p_\chi\otimes C^q_\psi\to C^{p+q}_{\chi\psi}$.
Finally, if $u$ is an $\F_2$-cocycle and $\widetilde u$ is its $0$--$1$ lift, then
\begin{equation}
 \beta_\chi[u]
 =\left[\frac{(B_{p+1}^\chi)^t\widetilde u}{2}\right].
 \label{eq:matrixbockstein}
\end{equation}
The numerator is even because $u$ is a cocycle modulo $2$.  If an integral lift
$\widetilde u$ is replaced by $\widetilde u+2v$, the displayed cocycle changes by
$(B_{p+1}^\chi)^tv$, which is the coboundary of $v$.  If the mod-$2$ cocycle $u$ is
changed by a coboundary, lift the bounding cochain integrally; the same calculation,
using $d^2=0$, again changes the displayed cocycle by an integral coboundary.  Hence
\eqref{eq:matrixbockstein} is independent of both choices.

Because every sign becomes $1$ modulo $2$, the reductions
$C^*_{\chi}(\Sigma)\otimes\F_2$ are canonically the same cochain complex; denote it by
$\overline C^*(\Sigma)$.  The cycle map needed below is also visible on the same
transport table.  Write $\ell_i:H=K_0\to\F_2$ for the restriction of the $i$th coordinate of
$(\Z/2)^r$.  For an edge $e\in\mathcal S_1$, let $h_{e,0},h_{e,1}\in H$ be the two
endpoint transports in \eqref{eq:transport}, and define
\begin{equation}
 \omega_i(e)=\ell_i(h_{e,0}+h_{e,1})\in\F_2.
 \label{eq:divisorcocycle}
\end{equation}
Here the signs $\varepsilon_{e,j}$ disappear modulo $2$.

\begin{proposition}
\label{prop:explicitcl2}
Each $\omega_i$ is a cocycle.  If $D_i$ is the invariant prime divisor associated
with $\rho_i$, then
\[
 [\omega_i]=\operatorname{cl}_2[D_i]
 \quad\text{in}\quad H^1(\overline C^*(\Sigma)).
\]
Consequently the substitution
\begin{equation}
 \operatorname{cl}^{\mathrm{fan}}_2:
 \Ch^*(X_\Sigma)=
 \frac{\F_2[x_1,\ldots,x_r]}{\overline I_{\mathrm{SR}}+\overline J_\Sigma}
 \longrightarrow H^*(\overline C^*(\Sigma)),
 \qquad x_i\longmapsto[\omega_i],
 \label{eq:fancl2}
\end{equation}
is the mod-$2$ real cycle class isomorphism.  On cochains, the image of a polynomial
is obtained by substituting the vectors \eqref{eq:divisorcocycle} and evaluating its
monomials with the Alexander--Whitney table \eqref{eq:AWmatrix}.
\end{proposition}

\begin{proof}
Fix the representative in $\mathcal S_0$ as a lift of each vertex of
$Y'/H$.  Along an edge $e$, the chosen lifts of its endpoints differ from the two
endpoints of the representative lift of $e$ by $h_{e,0}$ and $h_{e,1}$.  Thus
$h_{e,0}+h_{e,1}$ is the transition cocycle of the principal $H$-cover
$Y'\to Y'/H$.  Composition with $\ell_i$ is therefore a simplicial cocycle and
represents the first Stiefel--Whitney class of
\[
 Y'\times_H\mathbb R_{(-1)^{\ell_i}}\longrightarrow Y'/H.
\]
In the Cox quotient this associated line bundle is the real line bundle underlying
$\mathcal O(D_i)$ (replacing it by its dual does not change $w_1$).  If the chosen lift
of a quotient vertex $v$ is changed by $b_v\in H$, then $\omega_i$ changes by the
coboundary of the $0$-cochain $v\mapsto\ell_i(b_v)$.  Thus its cohomology class is
independent of all vertex and simplex representatives.  The coordinate section is
transverse and has zero locus $D_i(\mathbb R)$, so the Thom definition of the mod-$2$
cycle class gives $[\omega_i]=\operatorname{cl}_2[D_i]$.

The cycle class is multiplicative.  Hence the linear and Stanley--Reisner relations
in \eqref{eq:DJ} map to zero and \eqref{eq:fancl2} is well defined.  The ordinary
mod-$2$ real cycle map is a ring isomorphism for a smooth cellular variety by
\cite[Proposition~5.3]{RealCycle}; Proposition~\ref{prop:cellular} verifies its
hypotheses for $X_\Sigma$.
Finally, the simplicial cup product is precisely the transported
Alexander--Whitney formula, which proves the last assertion.  Equivalently, the
small-cover presentation of Davis--Januszkiewicz identifies the same classes as the
Stiefel--Whitney classes of the characteristic line bundles \cite[Theorem~4.14]{DavisJanuszkiewicz}.
\end{proof}

\begin{proposition}
\label{prop:certificate}
Put
\[
 q_\Sigma=\sum_{I\in K}2^{r-|I|}.
\]
The cubical face poset has $q_\Sigma$ elements, and its order complex has at most
$\sum_{j=1}^{n+1}q_\Sigma^j$ simplices.  Consequently the matrices
\eqref{eq:signedmatrix} and \eqref{eq:AWmatrix} can be enumerated after finitely many
operations, for all $2^{r-n}$ characters.  Formula \eqref{eq:divisorcocycle} adds
$r\lvert\mathcal S_1\rvert$ bits and makes the mod-$2$ cycle map explicit.  Smith
certificates for the matrices, together with the transport and divisor tables, certify
every additive group, product, reduction map, cycle map, and Bockstein occurring in
Theorem~\ref{thm:main}.
\end{proposition}

\begin{proof}
Formula \eqref{eq:cubicalcells} gives the first count.  A chain in the face poset has
length at most $n+1$, because $Y$ has dimension $n$, and the deliberately coarse
bound follows.  Freeness from Lemma~\ref{lem:freeness} makes every transport in
\eqref{eq:transport} unique.  Equations \eqref{eq:signedmatrix},
\eqref{eq:AWmatrix}, \eqref{eq:divisorcocycle}, and \eqref{eq:matrixbockstein} then
reduce all assertions to finite integer or mod-$2$ linear algebra.  Smith normal form supplies unimodular
change-of-basis matrices, so the output is checkable rather than merely existential.
\end{proof}

\begin{proposition}
\label{prop:executablecertificate}
The enumeration in Proposition~\ref{prop:certificate} can be encoded by the following
finite object $\mathfrak C_\Sigma$:
\begin{enumerate}[(i)]
\item a basis of $K_0$, the cubical cells \eqref{eq:cubicalcells}, and ordered orbit
representatives $\mathcal S_d$;
\item every face transport \eqref{eq:transport} and every signed matrix
$B_d^\chi$;
\item for every $p+q\leq n$, the front- and back-face transports in
\eqref{eq:AWmatrix}.
\item the $r$ divisor cocycles \eqref{eq:divisorcocycle}, together with the linear
and minimal-nonface relations they satisfy modulo coboundaries.
\end{enumerate}
This object reconstructs all cochain complexes and cochain products in
Theorem~\ref{thm:main}.  Validity relative to the input fan is decidable: first
re-enumerate (i)--(iv) from the rays and cones and require equality with every stored
field; then check the finite integer identities consisting of the
dimension checks, $B_{d-1}^\chi B_d^\chi=0$ for every $d$ and $\chi$, and
\[
 d(a\smile b)=da\smile b+(-1)^{|a|}a\smile db
\]
for every pair of characters and every pair of basis cochains.  In addition, one checks
that the vectors in (iv) are cocycles, that every row of the characteristic matrix gives
a coboundary, and that the product indexed by every minimal nonface is a coboundary.
Reduction modulo $2$, substitution in \eqref{eq:fancl2}, and formula
\eqref{eq:matrixbockstein} then reconstruct the cycle maps, reductions, and Bocksteins;
Smith transformations against the stored matrices certify the resulting cohomology groups.
\end{proposition}

\begin{proof}
Items (i)--(iv) are finite by Proposition~\ref{prop:certificate} and
Proposition~\ref{prop:explicitcl2}.  Formula
\eqref{eq:signedmatrix} reconstructs every differential.  For a simplex
$[c_0<\cdots<c_{p+q}]$, item (iii) identifies the orbit representatives and deck
transports of $[c_0<\cdots<c_p]$ and $[c_p<\cdots<c_{p+q}]$; evaluating the two
characters on those transports is exactly the Alexander--Whitney product.  Hence the
listed identities are precisely the chain-complex and derivation axioms on basis
elements.  Proposition~\ref{prop:explicitcl2} proves that (iv) is the required cycle
map; the stated relation tests verify its factorization through the displayed Chow
presentation.  The remaining assertions follow from \eqref{eq:matrixbockstein} and
Smith normal form.
\end{proof}

\paragraph{Supplementary verification.}
The supplementary material contains a zero-dependency implementation of the finite
fan-to-cochain construction and a checker that independently reconstructs its output
from the rays and cones.  It verifies $d^2=0$, the cochain Leibniz identity, and the
linear and Stanley--Reisner relations for the divisor cocycles; a negative test confirms
that altered fan data with unchanged matrices are rejected.  The current reference
files certify the fan-to-cochain and divisor-cycle layers.  They do not store
unimodular Smith transformations or a normalized presentation of the final fibre
product, which belong to the terminating algorithm proved above.  None of these
computations is used as a premise in the proofs.

\subsection{Real closed base fields}

For $R=\mathbb R$, the quotient $Y/H$ has the homotopy type of
$X_\Sigma(\mathbb R)$.  For a non-Archimedean real closed field, ordinary singular
cohomology is not the correct target.  We use semialgebraic cohomology in the sense of
Delfs \cite{Delfs}.

\begin{proposition}
\label{prop:semialgebraic}
Let $R$ be a real closed field and let $X_\Sigma/R$ be obtained from the fixed smooth
projective integral fan $\Sigma$.  Fix the rational barycentric triangulation and the
simplex representatives used to define $C^*_{\chi}(\Sigma)$.  For every character
$\chi$ there are induced isomorphisms, compatible with cup products, reductions, and Bocksteins,
\[
 H^p_\sa(X_\Sigma(R),\Z_\chi)
 \cong H^p(C^*_\chi(\Sigma)).
\]
In particular, these semialgebraic cohomology groups and their products are independent
of the chosen real closed field.
\end{proposition}

\begin{proof}
Projectivity realizes $\Sigma$ as the normal fan of a lattice polytope $P$.  Put
\[
 Y_R=(D^1_R,S^0_R)^K,\qquad D^1_R=[-1,1]_R,\quad S^0_R=\{-1,1\}.
\]
This is a finite cubical semialgebraic complex defined over $\mathbb Q$, and the proof
of Lemma~\ref{lem:freeness} applies over every real closed field, so that $K_0$ acts
freely on $Y_R$.  If $F(p)$
is the smallest face containing $p\in P(R)$, let $\overline G_{F(p)}\subseteq G/K_0$
be the image of the subgroup generated by the facets containing $F(p)$.  The
semialgebraic small cover is
\begin{equation}
 M_R(P,\lambda_2)=
 (P(R)\times G/K_0)/{\sim},\qquad
 (p,\bar g)\sim(p,\bar g')
 \Longleftrightarrow \bar g^{-1}\bar g'\in \overline G_{F(p)}.
 \label{eq:smallcover}
\end{equation}
The standard rational piecewise-linear moment-angle map identifies
$Y_R/K_0$ with \eqref{eq:smallcover}.  On each face it is the barycentric map to the
corresponding face of $P$ and is given by the same rational affine formulas over every
real closed field; the inverse is facewise rational piecewise linear.  This is the
small-cover construction of \cite[Sections~1--2]{DavisJanuszkiewicz} interpreted in
the semialgebraic category.

We next compare the small cover with the real toric variety without choosing
square roots in Cox orbits.  After replacing $P$ by a positive integral multiple, its
lattice points $m_1,\ldots,m_N$ give an equivariant closed embedding
$X_\Sigma\hookrightarrow\PP^{N-1}$.  Let
$X_\Sigma(R)_{\geq0}$ be the intersection with the nonnegative projective orthant,
namely the set of points admitting homogeneous coordinates $z_j\geq0$.  The formula
\begin{equation}
 \mu_R([z_1:\cdots:z_N])
 =\frac{\sum_{j=1}^N z_j^2m_j}{\sum_{j=1}^N z_j^2}
 \label{eq:realmoment}
\end{equation}
is defined over $\mathbb Q$ and takes values in $P(R)$.  Over $\mathbb R$, the
projective toric moment-map theorem identifies the restriction
$\mu_{\mathbb R,+}:X_\Sigma(\mathbb R)_{\geq0}\to P(\mathbb R)$ as a
face-preserving homeomorphism; the sign translates of the nonnegative part cover the
real toric variety, and the stabilizer above a point in the relative interior of a face
is generated by the characteristic vectors of the facets containing that face
\cite[Section~4.2]{Fulton}\cite[Sections~1--2]{DavisJanuszkiewicz}.

All sets and maps in this assertion have graphs given by finite Boolean combinations
of polynomial equalities and weak inequalities with rational coefficients.  The
statements that $\mu_{\mathbb R,+}$ is bijective, that the sign translates cover, and
that the displayed face stabilizers are exact are therefore first-order statements in
ordered fields and transfer from $\mathbb R$ to every real closed $R$.
The transferred inverse is semialgebraic and continuous: $P(R)$ is definably compact,
$X_\Sigma(R)_{\geq0}$ is semialgebraically closed in projective space, and a continuous
semialgebraic bijection from a definably compact space to a Hausdorff semialgebraic
space is a semialgebraic homeomorphism
\cite[Chapters~2 and~9]{BochnakCosteRoy}.  The same first-order stabilizer description
shows that
\[
 [p,\bar g]\longmapsto
 \bar g\cdot\mu_{R,+}^{-1}(p)
\]
is well defined precisely modulo the relation in \eqref{eq:smallcover}.  It is a
continuous semialgebraic bijection from the definably compact small cover to the
projective variety $X_\Sigma(R)$ and hence a semialgebraic homeomorphism.  Over
$\mathbb R$ this is also the real toric-space comparison of
\cite[Corollary~4.6]{Franz}.  We have therefore obtained a semialgebraic homeomorphism
\[
 Y_R/K_0\xrightarrow{\ \cong\ }X_\Sigma(R)
\]
whose defining graph is uniform over real closed fields.

Subdivide each copy of $[-1,1]$ at $0$ and take the barycentric subdivision of the
resulting cubical face poset.  Its
vertices, simplices, $K_0$-action, and quotient are all defined over $\mathbb Q$.
After interpreting the same rational affine formulas in $R$ and transporting them by
the preceding homeomorphism, it gives a finite semialgebraic triangulation of
$X_\Sigma(R)$.  A lift of a quotient simplex is
exactly an element of the representative set $\mathcal S_d$, and changing a face lift
is exactly the transport $h_{s,j}$ in \eqref{eq:transport}.  Cellular cochains with the
associated rank-one local system are therefore, on the nose, the matrices
$C^*_\chi(\Sigma)$.  The semialgebraic triangulation theorem with local coefficients
and the comparison of simplicial with semialgebraic cohomology identify this complex
with the left side of the proposition; invariance under subdivision is proved in
\cite[Chapters~II--III]{Delfs}.  Because the finite simplicial complex and every
transport element lie in the fixed group $K_0$, the resulting matrices are literally
independent of $R$.  Hornbostel--Wendt--Xie--Zibrowius explicitly record the use of
semialgebraic cohomology for twisted real realization over arbitrary real closed fields
in \cite[Remarks~1.1 and~5.10]{RealCycle}.

The comparison is thus induced by the displayed finite cochain identification, not
merely by an abstract equality of groups.  It commutes with the tensor map
$\Z_\chi\otimes\Z_\psi\to\Z_{\chi\psi}$, with
$\Z_\chi\to\F_2$, and with the connecting homomorphism of
\eqref{eq:topbocksteinsequence}.  This proves all claimed compatibilities, not only
the abstract group isomorphism.
\end{proof}

\section{The real cycle class map and the main theorem}
\label{sec:realcycle}

Let $R$ be real closed.  For a line bundle $L$ on a smooth $R$-variety $X$ and
$j\geq0$, the real cycle class map has the form
\[
 \operatorname{cl}_R:
 H^i(X,\I^j(L))\longrightarrow
 H^i_\sa(X(R),2^j\Z(L(R))).
\]
Here $2^j\Z(L(R))$ is the subsheaf of the orientation local system whose stalks are
$2^j\Z$.  For $j\geq0$, division by $2^j$ is the canonical isomorphism
$2^j\Z(L(R))\cong\Z(L(R))$; write
$\operatorname{cl}^{\mathrm{nor}}_R=2^{-j}\operatorname{cl}_R$ for the normalized
map.  This normalization is multiplicative because
$2^j\Z\otimes2^{j'}\Z\to2^{j+j'}\Z$ is ordinary multiplication.
Over $\mathbb R$ this is singular cohomology with the orientation local system of the
real line bundle.  Hornbostel--Wendt--Xie--Zibrowius prove compatibility with cup
products in Proposition~4.3 and Corollary~4.4, compatibility with the B\"ar and
Bockstein sequences in Proposition~4.14, and the cellular isomorphism in
Theorem~5.7 \cite{RealCycle}.  Their Remarks~1.1 and~5.10 record the corresponding
semialgebraic statements over arbitrary real closed fields.

Here and below, ``cellular'' in the real-cycle comparison means a filtrable
affine-space paving.  This condition is strictly stronger than the cellular
$\mathbb A^1$-structure used by Liu--Peng \cite{LiuPeng}, whose strata may contain
torus factors.

\begin{theorem}
\label{thm:realcycle}
Let $X/R$ be smooth and cellular in the following sense: it admits a finite filtration
$\varnothing=X_{-1}\subset X_0\subset\cdots\subset X_m=X$ by closed subvarieties such
that every stratum $X_i\setminus X_{i-1}$ is an affine space over $R$.  Then, for every
line bundle $L$ and all $j\geq i\geq0$, the normalized cycle class map is an isomorphism
\begin{equation}
 H^i(X,\I^j(L))
 \xrightarrow{\ \cong\ }
 H^i_\sa(X(R),\Z(L(R))).
 \label{eq:realcycleiso}
\end{equation}
In particular this applies on the diagonal $j=i=p$.  The isomorphisms are
multiplicative for tensor products of actual line bundles.  If
$\Pic(X)$ is free and $\mathfrak r$ is a basis rigidification, they therefore assemble
to an isomorphism of $\Z\oplus(\Z/2)^\varrho$-graded rigidified rings
\[
 \bigoplus_{p,\epsilon}H^p(X,\I^p(L_\epsilon))
 \xrightarrow{\ \cong\ }
 \bigoplus_{p,\epsilon}H^p_\sa(X(R),\Z(L_\epsilon(R))),
\]
where the products on both sides use the real realization of $\mathfrak r$.
Under these isomorphisms, $\rho_L$ corresponds to coefficient reduction
$\Z(L(R))\to\F_2$, and $\beta_L$ corresponds to the twisted integral Bockstein.
\end{theorem}

\begin{proof}
Put $\overline{\I}^{\,q}=\I^q/\I^{q+1}$ and
$\pi=\langle\!\langle-1\rangle\!\rangle\in I(R)$.  Since $R$ is real closed, signature
identifies
\[
 \W(R)\xrightarrow{\ \cong\ }\Z,
 \qquad I^q(R)\xrightarrow{\ \cong\ }2^q\Z,
\]
and carries multiplication by $\pi$ to multiplication by $2$.  In particular,
multiplication by $\pi$ is an isomorphism $I^q(R)\to I^{q+1}(R)$, and it induces the
corresponding isomorphism on $\overline{\I}^{\,q}(R)$, for every $q\geq0$.

We first record the stabilization step.  Homotopy invariance, localization, and
d\'evissage along a complementary open cellular filtration give, by induction on the
number of cells,
\begin{equation}
 \pi:H^i(X,\overline{\I}^{\,j})\xrightarrow{\ \cong\ }
 H^i(X,\overline{\I}^{\,j+1})
 \qquad (j\geq i).
 \label{eq:barIstabilization}
\end{equation}
This is precisely the argument of \cite[Proposition~5.4]{RealCycle}; it uses at the
base of the cellular induction only the coefficient isomorphism over $\Spec R$ above.
Comparing the B\"ar exact sequences for consecutive powers of $\I$, and descending
from the range in which multiplication by $\pi$ is already an isomorphism of sheaves,
then gives, for every line bundle $L$,
\begin{equation}
 \pi:H^i(X,\I^j(L))\xrightarrow{\ \cong\ }
 H^i(X,\I^{j+1}(L))
 \qquad (j\geq i).
 \label{eq:Istabilization}
\end{equation}
This is \cite[Proposition~5.5]{RealCycle}.  The same proof applies over $R$: a real
closed field satisfies the coefficient hypothesis in
\cite[Remark~5.10 and Lemma~5.11]{RealCycle}, and the required eventual sheaf
stabilization follows from $\operatorname{vcd}_2(R)=0$.

For fixed $i$ and $L$, write
\[
 A_j=H^i(X,\I^j(L)),
 \qquad
 B_j=H^i_\sa(X(R),2^j\Z(L(R))).
\]
Compatibility of the signature morphism with $\pi$ gives a commutative square
\begin{equation}
\begin{array}{ccc}
 A_j&\xrightarrow{\ \operatorname{cl}_{R,j}\ }&B_j\\
 \big\downarrow{\scriptstyle\cdot\pi}&&
 \big\downarrow{\scriptstyle\cdot2}\\
 A_{j+1}&\xrightarrow{\ \operatorname{cl}_{R,j+1}\ }&B_{j+1}.
\end{array}
\label{eq:cyclebackwardsquare}
\end{equation}
The left vertical arrow is an isomorphism for $j\geq i$ by
\eqref{eq:Istabilization}; the right vertical arrow is an isomorphism because
$2^j\Z(L(R))\xrightarrow{2}2^{j+1}\Z(L(R))$ is an isomorphism of local systems.
For all sufficiently large $j$, the horizontal signature map is an isomorphism by the
large-power comparison, including arbitrary line-bundle twists, used in the proof of
\cite[Theorem~5.7(a)]{RealCycle}.  The real-closed version is obtained by replacing
singular cohomology with semialgebraic cohomology as specified in
\cite[Remarks~1.1 and~5.10]{RealCycle}.  Backward induction in
\eqref{eq:cyclebackwardsquare} therefore proves that
\[
 H^i(X,\I^j(L))\xrightarrow{\ \cong\ }
 H^i_\sa(X(R),2^j\Z(L(R)))
\]
for every $j\geq i$.  Division by $2^j$ gives \eqref{eq:realcycleiso}.

The product comparison is compatible with arbitrary line-bundle twists by
\cite[Proposition~4.3 and Corollary~4.4]{RealCycle}.  Compatibility with reduction and
the B\"ar and Bockstein connecting maps is
\cite[Proposition~4.14]{RealCycle}; the semialgebraic coefficient sequence is exactly
$0\to\Z(L(R))\xrightarrow{2}\Z(L(R))\to\F_2\to0$ after normalization.
Compatibility with the square-cancellation maps of $\mathfrak r$ follows from
functoriality in the twisting line bundle and the canonical orientation of the square
of a real line bundle.  Finally, signature sends $\langle-1\rangle$ to $-1$, so
\eqref{eq:rigidcommutativity} is carried to the Koszul sign in the topological cup
product.  This proves the rigidified-ring statement with its stated commutativity
convention.  Proposition~\ref{prop:semialgebraic} then identifies the target, including
all products and structure maps, with the fixed finite fan model.
\end{proof}

We can now state the main result without any hidden completeness, freeness, diagonal,
or Steenrod-kernel assumptions.

\begin{theorem}
\label{thm:main}
Let $R$ be a real closed field and let $X_\Sigma/R$ be a smooth projective split toric
variety.  Choose a basis rigidification $\mathfrak r$ of $\Pic(X_\Sigma)$ as in
Definition~\ref{def:rigidification}.  For
$\epsilon\in(\Z/2)^\varrho$, let $\chi_\epsilon$ be the character associated by
Lemma~\ref{lem:characters} with $L_\epsilon$, and define
\[
 A_{\Sigma,\mathfrak r}=
 \bigoplus_{p,\epsilon}H^p(C^*_{\chi_\epsilon}(\Sigma)),
\]
with multiplication induced by the equivariant Alexander--Whitney diagonal.  Define
\begin{equation}
 K^p_\epsilon(\Sigma)=
 \{z\in\CH^p(X_\Sigma):
 \beta_{\chi_\epsilon}(\operatorname{cl}_2(\bar z))=0\},
 \label{eq:topK}
\end{equation}
where $\operatorname{cl}_2=\operatorname{cl}^{\mathrm{fan}}_2$ is the explicit
mod-$2$ cycle class isomorphism \eqref{eq:fancl2} and
$\beta_{\chi_\epsilon}$ is \eqref{eq:topbockstein}.  The two structure maps used in
the fibre product below are, in the sector labelled by $\epsilon$,
\begin{equation}
\begin{aligned}
 a_\epsilon:H^p(C^*_{\chi_\epsilon}(\Sigma))&\longrightarrow\Ch^p(X_\Sigma),
 &a_\epsilon&=(\operatorname{cl}^{\mathrm{fan}}_2)^{-1}\circ r_{\chi_\epsilon},\\
 b_\epsilon:K^p_\epsilon(\Sigma)&\longrightarrow\Ch^p(X_\Sigma),
 &b_\epsilon(z)&=\bar z,
\end{aligned}
\label{eq:mainstructuremaps}
\end{equation}
where in the first line we use the canonical identification
$H^p(Y/H;\F_2)=H^p(\overline C^*(\Sigma))$ from the common reduced cochain complex.
Both targets are placed in the $\epsilon$-labelled copy of
$\Ch^p(X_\Sigma)$.  Then
\begin{equation}
 \CHW^\bullet_{\mathfrak r}(X_\Sigma,\Pic/2)
 \cong
 A_{\Sigma,\mathfrak r}
 \times_{\Ch^*(X_\Sigma)[(\Z/2)^\varrho]}
 \bigoplus_{p,\epsilon}K^p_\epsilon(\Sigma)
 \label{eq:mainfibre}
\end{equation}
as a $\Z\oplus(\Z/2)^\varrho$-graded ring.  The multiplication on the last direct
sum is ordinary Chow multiplication with addition of the formal $\epsilon$-labels.

Every group, map, and product in \eqref{eq:mainfibre} is computable by a terminating
finite algorithm from the rays and cones of $\Sigma$ and the divisor representatives
of $\mathfrak r$.  If a rigidification is not supplied, a choice of lattice coordinates
and an ordering of the rays, followed by Hermite normal form in
\eqref{eq:divisorsequence}, produces one possible ordered Picard basis and
torus-invariant divisor representatives.  This auxiliary choice is not intrinsic to
the abstract fan, and no independence of it is asserted.  The algorithm produces the
integer matrices \eqref{eq:signedmatrix} and \eqref{eq:AWmatrix}, the Bockstein
formula \eqref{eq:matrixbockstein}, the divisor vectors
\eqref{eq:divisorcocycle}, and Smith certificates; its input-dependent size
bound is Proposition~\ref{prop:certificate}, and its raw cochain-algebra output admits
the independently checkable, fan-bound deterministic certificate of
Proposition~\ref{prop:executablecertificate}.
\end{theorem}

\begin{proof}
Proposition~\ref{prop:cellular} gives cellularity and torsion-free Chow groups; in
particular, $\Pic(X_\Sigma)=\CH^1(X_\Sigma)$ is free.  Theorem~\ref{thm:fibre}
therefore identifies the rigidified total Chow--Witt ring with
$A_{\mathfrak r}(X_\Sigma)\times_{C_{\mathfrak r}(X_\Sigma)}
K_{\mathfrak r}(X_\Sigma)$.
Theorem~\ref{thm:realcycle}, Proposition~\ref{prop:semialgebraic}, and
Lemma~\ref{lem:characters} identify the first factor, its products, and its reduction
map with $A_{\Sigma,\mathfrak r}$.  Compatibility with Bocksteins identifies the algebraic group
\eqref{eq:exactK} with \eqref{eq:topK}.  This proves \eqref{eq:mainfibre}.

For effectivity, compute \eqref{eq:DJ}, \eqref{eq:picmod2}, $K_0$, and its character
group from the fan.  Read the chosen basis and invariant-divisor representatives from
$\mathfrak r$; alternatively, after making the auxiliary coordinate and ordering
choices just stated, Smith or Hermite normal form for \eqref{eq:divisorsequence}
constructs one admissible $\mathfrak r$.  Lemma~\ref{lem:freeness} and
Proposition~\ref{prop:equivdiag} give finite free signed cochain complexes with a
finite diagonal.  Corollary~\ref{cor:effective} computes their cohomology, products,
reductions, and Bocksteins.  Proposition~\ref{prop:explicitcl2} evaluates
$\operatorname{cl}_2$ on the divisor presentation by finite cup products.  Finally,
ordinary integer linear algebra computes the inverse images in \eqref{eq:topK} and the
fibre product.  The explicit enumeration
and its checkable certificates are given in Propositions~\ref{prop:certificate}
and~\ref{prop:executablecertificate}.
\end{proof}

The degree-$(0,0)$ component specializes to the expected coefficient ring in the
present toric setting.  Indeed, the characteristic vectors of the facets meeting any
vertex form a basis of $(\Z/2)^n$, so the small cover is connected.  Thus its untwisted
degree-zero cohomology is $\Z$, while $\CH^0(X_\Sigma)=\Z$ and the degree-zero
Bockstein vanishes.  The degree-zero part of \eqref{eq:mainfibre} is consequently
\[
 \Z\times_{\F_2}\Z
 =\{(s,r):s\equiv r\pmod2\}\cong\GW(R),
\]
where the last isomorphism is given by signature and rank.  This equality uses both
connectedness and the real-closed hypothesis and was not assumed in the general
discussion following Definition~\ref{def:rigidification}.

\begin{corollary}
\label{cor:sqshortcut}
In the setting of Theorem~\ref{thm:main}, fix $p$ and $\epsilon$, and put
$L=L_\epsilon$.  If reduction
\[
 H^{p+1}(C^*_{\chi_\epsilon}(\Sigma))\longrightarrow
 H^{p+1}_\sa(X_\Sigma(R);\F_2)
\]
is injective on $\im\beta_{\chi_\epsilon}$, then
\[
 K_\epsilon^p(\Sigma)=
 \{z\in\CH^p(X_\Sigma):
 (\operatorname{Sq}^2+c_1(L))\bar z=0\}.
\]
The injectivity test is itself decidable from the finite signed cochain matrices.
\end{corollary}

\paragraph{Algorithmic form.}
The proof yields the following concrete workflow.
\begin{enumerate}[leftmargin=*,label=\textbf{Step \arabic*.}]
\item Compute $\CH^*(X_\Sigma)$ and $\Pic(X_\Sigma)/2$ from
      \eqref{eq:DJ} and \eqref{eq:picmod2}; record the ordered basis and
      invariant-divisor representatives supplied by $\mathfrak r$.  If only a coordinate
      presentation of the fan is given, first choose a ray order and use
      \eqref{eq:divisorsequence} to construct an auxiliary rigidification.
\item Construct $\RZ_K$, $K_0=\ker\lambda_2$, an invariant subdivision, and all
      character complexes \eqref{eq:signedcomplex}.
\item Compute the edge vectors $\omega_i$ from \eqref{eq:divisorcocycle}; evaluate
      $\operatorname{cl}_2$ on the Chow presentation by Alexander--Whitney products.
\item Use Smith normal form and the equivariant diagonal to compute
      $A_{\Sigma,\mathfrak r}$, all
      products, reductions, and Bocksteins.
\item Apply each $\beta_{\chi_\epsilon}$ to the computed cycle-class vectors and take the
      integral inverse images of its kernel to obtain
      $K_\epsilon^p(\Sigma)$.  Use the Steenrod shortcut only after the test in
      Corollary~\ref{cor:sqshortcut} succeeds.
\item Form the fibre product sector by sector and simplify it to generators and
      relations or to multiplication tables.
\end{enumerate}

No check that ``setting $\eta=0$ gives the Chow ring'' is used as a substitute for a
proof.  The forgetful map is built into the fibre product from the start.

\section{Hirzebruch surfaces}
\label{sec:hirzebruch}

We now execute every step for the Hirzebruch surface $\mathbb F_a$, $a\geq0$.  This
provides a complete all-twist example and verifies directly when the Steenrod shortcut
is valid.

\subsection{Fan, Chow ring, and Bockstein kernels}

Take the cyclically ordered rays
\[
 v_1=(0,1),\quad v_2=(1,0),\quad
 v_3=(0,-1),\quad v_4=(-1,a),
\]
with maximal cones $12,23,34,41$.  If $s=[D_1]$ and
$f=[D_2]=[D_4]$, then
\begin{equation}
 \CH^*(\mathbb F_a)=\Z[s,f]/(f^2,s(s+af)),
 \qquad \Pic(\mathbb F_a)/2=\{0,s,f,s+f\}.
 \label{eq:FaChow}
\end{equation}
Let $\mathfrak r_a$ be the basis rigidification associated with the ordered Picard
basis $(s,f)$ and the representatives
$\mathcal B_s=\mathcal O(D_1)$ and $\mathcal B_f=\mathcal O(D_2)$.  Thus the sector
$(\epsilon,\delta)$ is represented by
$L_{\epsilon,\delta}=\mathcal B_s^{\otimes\epsilon}\otimes
\mathcal B_f^{\otimes\delta}$, with multiplication fixed by
Definition~\ref{def:rigidification}.  All ring statements in this section use
$\mathfrak r_a$.
We use the same letters for their mod-$2$ reductions.  Thus
\[
 \Ch^*(\mathbb F_a)=
 \F_2[s,f]/(f^2,s^2+a sf).
\]

Let $w=\epsilon s+\delta f$ and $z=us+vf$ with
$\epsilon,\delta,u,v\in\F_2$.  On divisors, the motivic square is the square, so
\begin{equation}
 \operatorname{Sq}^2_w(z)
 =z^2+wz
 =\bigl(\epsilon v+(a(1+\epsilon)+\delta)u\bigr)sf.
 \label{eq:FaSq}
\end{equation}

\subsection{Deck transformations and the twist labels}

Put $\bar a=a\bmod2$.  Solving $\lambda_2(g)=0$ for the displayed ray matrix gives
\begin{equation}
 K_0=\{(p+\bar a q,q,p,q):p,q\in\F_2\}.
 \label{eq:FaK0deck}
\end{equation}
Choose the ordered basis
\[
 x=(\bar a,1,0,1),\qquad y=(1,0,1,0)
\]
and choose the fundamental-group generators below so that their monodromies in the
moment-angle cover are $x$ and $y$.  Lemma~\ref{lem:characters} gives the complete
character table
\begin{equation}
\begin{array}{c|cc|cc}
 &\chi_s(x)&\chi_s(y)&\chi_f(x)&\chi_f(y)\\ \hline
 a\text{ even}&+1&-1&-1&+1\\
 a\text{ odd}&-1&-1&-1&+1.
\end{array}
\label{tab:Facharacters}
\end{equation}
If $x^*,y^*$ are the dual mod-$2$ degree-one cellular classes, this says
\begin{equation}
 x^*=f,\qquad
 y^*=\begin{cases}s,&a\text{ even},\\s+f,&a\text{ odd}.
 \end{cases}
 \label{eq:Facochainlabels}
\end{equation}
In particular, for odd $a$ the table reads
\[
 \chi_0=(1,1),\quad \chi_s=(-1,-1),\quad
 \chi_f=(-1,1),\quad \chi_{s+f}=(1,-1)
\]
on $(x,y)$.  This fixes the twist labels in all matrices and products below; no
unrecorded permutation of $s$ and $f$ is being used.

\subsection{The four signed cellular complexes}

The real locus $M_a=\mathbb F_a(R)$ has the semialgebraic homotopy type
\[
 M_a\simeq
 \begin{cases}
 T^2,&a\text{ even},\\
 \text{the Klein bottle},&a\text{ odd}.
 \end{cases}
\]
The four characters are labelled by $w\in\{0,s,f,s+f\}$.

For the torus, use the one-vertex cell structure associated with
$\pi_1(T^2)=\langle x,y\mid xyx^{-1}y^{-1}\rangle$.  If a character takes the values
$(\alpha,\gamma)\in\{\pm1\}^2$ on $(x,y)$, the signed cochain complex is
\begin{equation}
 0\longrightarrow\Z
 \xrightarrow{d^0_{\alpha,\gamma}}\Z^2
 \xrightarrow{d^1_{\alpha,\gamma}}\Z
 \longrightarrow0,
 \label{eq:toruscomplex}
\end{equation}
where
\[
 d^0(n)=((\alpha-1)n,(\gamma-1)n),\qquad
 d^1(u,v)=(1-\gamma)u+(\alpha-1)v.
\]
Smith normal form gives
\begin{equation}
\begin{array}{c|ccc}
 a\text{ even}&H^0_w&H^1_w&H^2_w\\ \hline
 w=0&\Z&\Z^2&\Z\\
 w\ne0&0&\Z/2&\Z/2.
\end{array}
\label{tab:even-groups}
\end{equation}

For the Klein bottle, use
\[
 \pi_1(M_a)=\langle x,y\mid xyx^{-1}=y^{-1}\rangle.
\]
Fox differentiation of $xyx^{-1}y$ gives
\begin{equation}
 0\longrightarrow\Z
 \xrightarrow{d^0_{\alpha,\gamma}}\Z^2
 \xrightarrow{d^1_{\alpha,\gamma}}\Z
 \longrightarrow0,
 \label{eq:Kleincomplex}
\end{equation}
with
\[
 d^0(n)=((\alpha-1)n,(\gamma-1)n),\qquad
 d^1(u,v)=(1-\gamma)u+(\alpha+\gamma)v.
\]
By \eqref{tab:Facharacters}, the character $f$ is
$(\alpha,\gamma)=(-1,1)$; it is the orientation character, while $s=(-1,-1)$ and
$s+f=(1,-1)$.  The four Smith forms give
\begin{equation}
\begin{array}{c|ccc}
 a\text{ odd}&H^0_w&H^1_w&H^2_w\\ \hline
 w=0&\Z&\Z&\Z/2\\
 w=s&0&\Z/2&\Z/2\\
 w=f&0&\Z\oplus\Z/2&\Z\\
 w=s+f&0&\Z/2&\Z/2.
\end{array}
\label{tab:odd-groups}
\end{equation}
Here and below $H^p_w$ abbreviates
$H^p(M_a;\Z_w)\cong H^p(\mathbb F_a,\I^p(w))$.

\subsection{Reduction maps}

Modulo $2$, both differentials in \eqref{eq:toruscomplex} and
\eqref{eq:Kleincomplex} vanish.  We identify their two degree-one cochains with
$x^*,y^*$ by \eqref{eq:Facochainlabels}; direct reduction of the integral cocycles
then gives the following images.

\begin{lemma}
\label{lem:Fareductions}
In degree zero, only $H^0_0\to\F_2$ is nonzero.  Every degree-two group in
\eqref{tab:even-groups} and \eqref{tab:odd-groups} maps onto the class $sf$.  In
degree one,
\begin{equation}
\begin{array}{c|cccc}
 &0&s&f&s+f\\ \hline
 a\text{ even}&\langle s,f\rangle&\langle s\rangle&
 \langle f\rangle&\langle s+f\rangle\\
 a\text{ odd}&\langle f\rangle&\langle s\rangle&
 \langle s,f\rangle&\langle s+f\rangle.
\end{array}
\label{tab:reductionimages}
\end{equation}
\end{lemma}

\begin{proof}
For the torus, the kernels and images of the two displayed matrices show that the
unique nonzero class in $H^1_w\cong\Z/2$ reduces to $w$ for $w\ne0$.
For the Klein bottle, the untwisted kernel of $d^1$ is generated by the first
degree-one cochain, which we label $f$.  In the orientation sector $w=f$, the first
cochain is torsion and reduces to $f$.  The sum of the first and second cochains is a
free generator and reduces to $x^*+y^*=s$.
The $s$- and $(s+f)$-sectors follow by the same two-by-two calculation.  The
degree-two assertions follow from the cokernels of $d^1$.
\end{proof}

Exactness of \eqref{eq:topbocksteinsequence} says that the images in
Lemma~\ref{lem:Fareductions} are exactly the kernels of the four twisted Bocksteins.
Every Bockstein image lies in the $2$-torsion subgroup.  Reduction is injective on
this subgroup in every sector.  It is an isomorphism on the $2$-torsion when the
target is $\Z/2$, and the target has no $2$-torsion when it is $\Z$.  The only mixed
target is $H^1_f\cong\Z\oplus\Z/2$ for odd $a$; its torsion summand is generated by
$\ell_f$, whose reduction is the nonzero class $f$ by
Lemma~\ref{lem:Fareductions}.  Lemma~\ref{lem:sqcriterion} therefore applies in all
bidegrees on $\mathbb F_a$.

\begin{proposition}
\label{prop:FaK}
For $w=\epsilon s+\delta f$, the groups in the right factor of
Theorem~\ref{thm:fibre} are
\begin{align}
 K^0_{\epsilon,\delta}
 &=\begin{cases}
 \Z,&(\epsilon,\delta)=(0,0),\\
 2\Z,&(\epsilon,\delta)\ne(0,0),
 \end{cases}
 \label{eq:FaK0}\\
 K^1_{\epsilon,\delta}
 &=\{ms+nf:\
 \epsilon n+(a(1+\epsilon)+\delta)m\equiv0\pmod2\},
 \label{eq:FaK1}\\
 K^2_{\epsilon,\delta}&=\Z\,sf.
 \label{eq:FaK2}
\end{align}
\end{proposition}

\begin{proof}
The preceding injectivity argument identifies the exact Bockstein kernel with the
kernel of \eqref{eq:FaSq}.  In degree zero,
$\operatorname{Sq}^2_w(1)=w$, giving \eqref{eq:FaK0}.  Formula
\eqref{eq:FaSq} gives \eqref{eq:FaK1}; degree three vanishes, giving
\eqref{eq:FaK2}.
\end{proof}

\begin{lemma}
\label{lem:surfaceproducts}
Let $u\in H^1_w$ and $v\in H^1_{w'}$.  If
$H^2_{w+w'}\cong\Z/2$, then $uv$ is uniquely determined by the mod-$2$ product of
their reductions.  If $H^2_{w+w'}\cong\Z$ and one factor is torsion, then $uv=0$.
After these cases, the only undetermined products of positive-degree generators are
the free orientation pairings; graded commutativity and Poincar\'e duality fix them up
to the single choice of a fundamental class.
\end{lemma}

\begin{proof}
Lemma~\ref{lem:Fareductions} identifies reduction from every $\Z/2$ top-degree
sector with the nonzero subgroup $\F_2sf$, so it is injective.  This proves the first
claim.  The image of a product containing a torsion factor is torsion, and hence is
zero in a free abelian target.  Poincar\'e duality with local coefficients identifies
the remaining free--free pairing with evaluation on the orientation class.  On the
orientable torus, graded commutativity also gives $2u^2=0$ in a free top group, hence
$u^2=0$.  These observations exhaust all degree-two products because $M_a$ is a
closed surface.
\end{proof}

\subsection{Complete multiplication in the signed ring}

We record the multiplication rather than leaving it implicit in a list of groups.
For even $a$, choose free untwisted degree-one classes $S,F$ reducing to $s,f$ and
an orientation $T\in H^2_0$ with $SF=T$.  For each nonzero $w$, let
$\ell_w\in H^1_w$ be the unique nonzero class and let
$t_w\in H^2_w$ be the unique nonzero class.  Then
\begin{equation}
 S^2=F^2=0,\quad SF=-FS=T,
 \label{eq:evenfreeprod}
\end{equation}
and every product landing in a nonzero twist sector is the unique class whose
reduction is the product in
$\F_2[s,f]/(s^2,f^2)$.  Every product of two positive-degree classes landing in the
untwisted free group $H^2_0=\Z$ is zero unless it is the pairing in
\eqref{eq:evenfreeprod}.  Lemma~\ref{lem:surfaceproducts} proves that these rules are
complete.  They give, for example,
\[
 F\ell_s=t_s,\quad S\ell_f=t_f,\quad
 \ell_s\ell_f=t_{s+f},\quad
 \ell_s\ell_{s+f}=t_f,
\]
and all squares $\ell_w^2$ vanish.

For odd $a$, choose
\[
 U\in H^1_0,\quad V\in H^1_f,\quad T_f\in H^2_f
\]
with reductions $f,s,sf$ and orientation pairing
\begin{equation}
 UV=T_f,\qquad VU=-T_f.
 \label{eq:oddorientationpairing}
\end{equation}
Let $\ell_s,\ell_f,\ell_{s+f}$ denote the torsion degree-one generators reducing to
$s,f,s+f$, respectively, where $\ell_f$ is the torsion summand of $H^1_f$.  Let
$t_w$ be the generator of $H^2_w\cong\Z/2$ for $w\ne f$, including $w=0$.

\begin{proposition}
\label{prop:Kleinring}
Besides the unit action and \eqref{eq:oddorientationpairing}, multiplication is
determined by the following rule: every product landing in one of the groups
$H^2_w\cong\Z/2$, $w\ne f$, is the unique element whose reduction is the product in
\[
 \F_2[s,f]/(f^2,s^2+sf),
\]
and every product involving a torsion factor and landing in the free group
$H^2_f=\Z$ is zero.  In particular,
\begin{align*}
 &V^2=t_0,\quad U^2=0,\quad U\ell_s=t_s,
 \quad U\ell_f=0,\quad U\ell_{s+f}=t_{s+f},\\
 &V\ell_s=t_{s+f},\quad V\ell_f=t_0,
 \quad V\ell_{s+f}=0,\\
 &\ell_s^2=t_0,\quad \ell_f^2=0,
 \quad \ell_{s+f}^2=t_0,\\
 &\ell_s\ell_f=t_{s+f},\quad
 \ell_s\ell_{s+f}=0,\quad
 \ell_f\ell_{s+f}=t_s.
\end{align*}
These identities, graded commutativity, and addition of twist labels give every
product in $\bigoplus_{p,w}H^p_w$.
\end{proposition}

\begin{proof}
Poincar\'e duality with local coefficients gives the unimodular pairing
\eqref{eq:oddorientationpairing}; the orientation character is $f$.  A product with a
torsion factor cannot be nonzero in the free group $H^2_f$.  For every other
degree-two sector, Lemma~\ref{lem:Fareductions} identifies reduction
$H^2_w\to\F_2sf$ with an isomorphism.  Hence reduction determines the product.
Multiplying the reductions of $U,V,\ell_s,\ell_f,\ell_{s+f}$ by the relations
$f^2=0$ and $s^2=sf$ gives the displayed list.  This is precisely the exhaustive
case distinction of Lemma~\ref{lem:surfaceproducts}.
\end{proof}

The analogous even-parity rules and Proposition~\ref{prop:Kleinring} make the entire
left factor in Theorem~\ref{thm:main} explicit, not merely its additive groups.

\subsection{The rigidified total Chow--Witt ring}

Let $A_a=\bigoplus_{p,w}H^p_w$ with the multiplication just described, and let
$K_a=\bigoplus_{p,\epsilon,\delta}K^p_{\epsilon,\delta}$ with
\eqref{eq:FaK0}--\eqref{eq:FaK2}.  If
$x\in K^p_{\epsilon,\delta}$ and
$y\in K^q_{\epsilon',\delta'}$, their product in $K_a$ is the ordinary Chow product
$xy$, placed in
$K^{p+q}_{\epsilon+\epsilon',\delta+\delta'}$.  This lies in the asserted subgroup
because $K_a$ is the rigidified forgetful image in Theorem~\ref{thm:fibre}; equivalently,
the defining Bockstein-kernel condition is multiplicatively closed.  The maps from
$A_a$ to mod-$2$ Chow theory are
given by Lemma~\ref{lem:Fareductions}, and the maps from $K_a$ are ordinary
reduction.

\begin{theorem}
\label{thm:Fa}
Let $R$ be real closed and $a\geq0$.  There is an isomorphism of
$\Z\oplus(\Z/2)^2$-graded rings
\begin{equation}
 \CHW^\bullet_{\mathfrak r_a}(\mathbb F_a,\Pic/2)
 \cong
 A_a\times_{\F_2[s,f]/(f^2,s^2+a sf)[(\Z/2)^2]}K_a.
 \label{eq:FaFinal}
\end{equation}
The additive groups of $A_a$ are \eqref{tab:even-groups} and
\eqref{tab:odd-groups}; all of its products are given by
\eqref{eq:evenfreeprod} and Proposition~\ref{prop:Kleinring}; all reduction maps are
given by \eqref{tab:reductionimages}; and $K_a$ is given by
\eqref{eq:FaK0}--\eqref{eq:FaK2}, with the restricted Chow multiplication specified
above.  Thus \eqref{eq:FaFinal} specifies every group,
product, twist, and forgetful map.
\end{theorem}

\begin{proof}
This is Theorem~\ref{thm:main} with the two finite signed complexes reduced to the
one-vertex torus and Klein-bottle complexes.  Propositions~\ref{prop:FaK} and
\ref{prop:Kleinring} provide the two factors and all structure maps explicitly.
\end{proof}

As a consistency check, the untwisted groups are
\[
 \CHW^0(\mathbb F_a)=\GW(R),
\]
\[
 \CHW^1(\mathbb F_a)\cong
 \begin{cases}
 \GW(R)^{\oplus2},&a\text{ even},\\
 2\Z\oplus\GW(R),&a\text{ odd},
 \end{cases}
\qquad
 \CHW^2(\mathbb F_a)\cong
 \begin{cases}
 \GW(R),&a\text{ even},\\
 \Z,&a\text{ odd}.
 \end{cases}
\]
Here $2\Z$ records an index-two forgetful image; abstractly it is infinite cyclic.

\subsection{Computational supplement}

The supplement contains two complementary audits using only the Python standard
library.  The general audit constructs the signed complexes, products, and divisor
cocycles for representatives of the two parities of $a$, while an independent checker
reconstructs these data from the fan and verifies the defining identities.  A second,
small-matrix audit starts from the four rays and the one-vertex torus or Klein-bottle
presentations.  It checks the character table, the cohomology and reduction images in
all twist sectors, every mod-$2$ product used in
Proposition~\ref{prop:Kleinring}, and the parity kernels in \eqref{eq:FaK1}.  It also
compares the general fan certificates with the explicit Hirzebruch calculation at the
integral free-rank and mod-$2$ ring levels.  All assertions in this section are proved
independently of these audits.

\section{The arithmetic boundary}
\label{sec:boundary}

The finite signed complex records real signatures.  It does not record the full Witt
theory of a non-real-closed field.  The smallest toric example already detects the
difference.

\begin{proposition}
\label{prop:P1obstruction}
Let $k$ be a perfect ordered field of characteristic different from $2$, let $R$ be a
real closure for the chosen ordering, and consider $\PP^1_k$.  Then
\begin{align*}
 H^1(\PP^1_k,\I^1)&\cong\W(k),\\
 H^1(\PP^1_k,\I^1(\mathcal O(1)))
 &\cong\W(k)/I(k)\cong\Z/2.
\end{align*}
Under base change and real realization these map to
\[
 H^1(S^1;\Z)\cong\Z,
 \qquad
 H^1(S^1;\Z_{\mathrm{sign}})\cong\Z/2.
\]
The second map is an isomorphism, while the first is the signature
$\W(k)\to\W(R)\cong\Z$.  Hence the signed cochain model over $R$ computes the full
diagonal $\I$-cohomology over $k$ only if the relevant signature maps are
isomorphisms.  In particular, it fails for $k=\mathbb Q$.
\end{proposition}

\begin{proof}
The untwisted group is the $n=1$ case of Fasel's untwisted projective bundle formula
\cite[Theorem~4.1]{FaselPBT}.  For the twisted group, the twisting morphism of
\cite[Definition~5.1]{FaselPBT}, its split injectivity
\cite[Corollary~5.8]{FaselPBT}, and the reduced groups of
\cite[Definition~5.9]{FaselPBT} reduce the calculation to the even-rank projective
bundle formula \cite[Theorem~8.1]{FaselPBT}.  For the projectivization of the trivial
rank-two bundle, the reduced $\mathcal O(-1)$-twisted term vanishes and the remaining
degree-one term is $\overline{\I}^{\,0}(k)=\W(k)/I(k)$.  The twists
$\mathcal O(1)$ and $\mathcal O(-1)$ differ by the square
$\mathcal O(1)^{\otimes2}$, so square cancellation identifies the convention in that
formula with the one used here.  The real line bundle underlying $\mathcal O(1)$ is the
nontrivial line bundle on $\mathbb RP^1=S^1$.  The one-cell cochain complex for the
trivial local system has zero differential and degree-one cohomology $\Z$; for the
sign local system the differential is multiplication by $2$, giving $\Z/2$.
The comparison in the untwisted sector restricts on coefficients to the chosen
signature.  Over $\mathbb Q$, for example, the anisotropic form
$\langle1,-2\rangle$ has signature zero but a nonzero Witt class, so the comparison is
not injective.
\end{proof}

\begin{corollary}
\label{cor:noscalar}
There is no construction obtained solely by replacing the coefficient ring $\Z$ in
the signed real cellular complex by $\W(k)$ that computes
$\bigoplus_L H^*(X_\Sigma,\I^*(L))$ for all perfect fields $k$.  In the
$\mathcal O(1)$-sector of $\PP^1$ such a replacement gives
$\W(k)/2\W(k)$, whereas the algebraic group is always
$\W(k)/I(k)\cong\Z/2$.
\end{corollary}

\begin{proof}
The signed cellular differential on the circle is multiplication by $2$.  The claimed
cohomology is therefore $\W(k)/2\W(k)$.  This need not have two elements; for
example, for a finite field $k=\F_q$ with $q\equiv1\pmod4$ the Witt ring is
$\Z/2\oplus\Z/2$.  Independently of any ordering, the twisted projective bundle formula
\cite[Definition~5.1 and Theorem~8.1]{FaselPBT} gives the algebraic answer
$\W(k)/I(k)\cong\Z/2$.  Thus this example does not invoke
Proposition~\ref{prop:P1obstruction}, whose hypotheses deliberately require an ordered
field.
\end{proof}

\begin{corollary}
\label{cor:fanonlyboundary}
There is no graded-ring construction depending only on an integral fan $\Sigma$ that,
even after a fixed basis rigidification $\mathfrak r$, computes
$\CHW^\bullet_{\mathfrak r}(X_\Sigma,\Pic/2)$ over every perfect field of characteristic
different from $2$.  This already fails in bidegree $(0,0)$ for the fan of $\PP^1$.
Moreover, the naive operation of replacing $\Z$ in the signed real cellular complex
by the abstract coefficient ring $\W(k)$ does not repair the failure in positive degree.
\end{corollary}

\begin{proof}
For $\PP^1_k$, the coefficient bidegree is
$\CHW^0(\PP^1_k)=\GW(k)$.  The integral fan is unchanged by base field, whereas
$\GW(\mathbb C)\cong\Z$ and $\GW(\mathbb R)\cong\Z\oplus\Z$ as abelian groups.  Thus
the fan alone cannot determine even the coefficient group.  Corollary~\ref{cor:noscalar}
shows the stated, deliberately narrower, positive-degree assertion: merely replacing
$\Z$ by $\W(k)$ produces the wrong $\mathcal O(1)$-twisted group.  We do not rule out
a construction that also retains the fundamental-ideal filtration, orderings, and
quadratic residue maps.
\end{proof}

The two failures are logically distinct.  Degree zero forces a general-field theory to
retain the quadratic coefficient ring.  The $\mathcal O(1)$-sector then forces it to
retain how the fundamental ideal and the twist enter the differential.  In higher
dimension these are the quadratic Rost--Schmid residue maps; they are arithmetic input,
not combinatorics recoverable from the fan.

\end{document}